\documentclass[a4paper, 10pt, conference]{IEEEconf}

\usepackage{amsmath,amsfonts,amssymb,amsthm,commath,graphicx, epstopdf, lmodern, txfonts}
\usepackage{graphicx, caption, subcaption}
\usepackage[latin1]{inputenc}
\usepackage[english]{babel}

\newtheorem{lem}{Lemma}
\newtheorem{thm}{Theorem}

\newcommand{\tr}[1]{\text{Tr}\left(#1\right)}

\newcommand{\trr}[1]{\text{Tr}^2\left(#1\right)}
\newcommand{\EE}[1]{{\text{\normalsize$\mathbb E$}}\left(#1\right)}
\newcommand{\PP}[1]{{\text{\normalsize$\mathbb P$}}\left(#1\right)}

\newcommand{\Mmu}{M_\mu}

\newcommand{\Id}{I_{\mathcal{H}}}

\newcommand{\K}{\boldsymbol{K}}
\newcommand{\X}{\boldsymbol{\Xi}}

\newcommand{\heta}{\overline{\eta}}
\newcommand{\hp}{\overline{p}}

\newcommand{\rhob}{\overline{\rho}}
\newcommand{\xib}{\overline{\xi}}

\newcommand{\dy}{{\dif{y}}}
\newcommand{\dt}{{\dif{t}}}
\newcommand{\Dt}{{\Delta t}}
\newcommand{\Dy}{\Delta{y}}
\newcommand{\RR}{\mathbb{R}}
\newcommand{\Ito}{It$\bar{\text{o}}$}
\newcommand{\nmax}{N}

\title{Parameter estimation  from  measurements  along quantum trajectories\footnote{This work has been partly funded by the Idex PSL * under the grant ANR-10-IDEX-0001-02 PSL *, by the Emergences program of Ville de Paris under the grant Qumotel and by the Projet Blanc ANR-2011-BS01-017-01 EMAQS.}}

\author{
P. Six \footnote{ Centre Automatique et Syst\`{e}mes, Mines-ParisTech, PSL Research University.
60, bd Saint-Michel 75006 Paris.}
\and
 Ph. Campagne-Ibarcq\footnote{Laboratoire Pierre Aigrain, Ecole Normale Sup\'erieure-PSL Research University, CNRS, Universit\'e Pierre et Marie Curie-Sorbonne Universit\'es, Universit\'e Paris Diderot-Sorbonne Paris Cit\'e, 24 rue Lhomond, 75231 Paris Cedex 05, France}
 \and
 L. Bretheau\footnotemark[3]
 \and
 B. Huard\footnotemark[3]
 \and
P.  Rouchon \footnotemark[2]}%

\begin{document}
\maketitle
\begin{abstract}

The dynamics of many open quantum systems are  described by stochastic master equations. In the discrete-time case, we recall the structure of the derived
 quantum filter governing the evolution of the density operator  conditioned to the measurement outcomes. We then describe the structure of the corresponding particle quantum filters for estimating  constant parameter and we  prove their stability.  In the  continuous-time (diffusive)  case,  we propose a new  formulation of  these particle quantum filters. The interest of this new formulation is first to prove stability, and also to provide an efficient algorithm preserving,  for any discretization step-size, positivity of the quantum  states and parameter classical probabilities. This algorithm is tested on experimental data  to estimate  the detection efficiency  for  a superconducting qubit whose fluorescence field is measured using a heterodyne detector.

\end{abstract}
\section{Introduction}

Parameter estimation in hidden Markov models is a well established subject (see, e.g., \cite{MoulinesBook2005}).  Twenty years ago Mabuchi \cite{Mabuc1996QaSOJotEOSPB}
has proposed  maximum likelihood methods    to estimate Hamiltonian parameters.  Later on, Gambetta and Wiseman~\cite{GambeW2001PRA} have given  a first  formulation of particle filtering techniques for classical parameter estimation in open quantum systems. This formulation has been analyzed  in~\cite{ChaseG2009PRA} via  an embedding   in the standard quantum filtering formalism.
Recently Negretti and M\o lmer~\cite{NegreM2013NJoP} have exploited this embedding to derive the general equations of a particle quantum  filter for systems governed by stochastic master equations  driven by Wiener processes (diffusive case). In these contributions,  realistic simulations illustrate the interest of such filters  for the  estimation of continuous parameters. In~\cite{KatoY2013}, similar  filters are used  for purely discrete parameters in order to discriminate between different topologies of  quantum networks. The Bayesian  parameter estimation used in the  measurement-based feedback experiment reported  in~\cite{Brakhane2012}  is in fact  a special case of  particle quantum filtering when  the quantum states remain diagonal in the energy-level basis, reduce to populations and  classical probabilities.

The contribution of this paper is twofold: with theorem~\ref{thm:stabilityGeneral}, we show that particle quantum filters are  always stable processes; with lemma~\ref{lem:SMEfilterPiPartialKraus}, we propose and justify   a new positivity preserving formulation in the diffusive case.  This formulation is shown to provide an efficient algorithm  for precisely estimating the detection efficiency from experimental  heterodyne measurements of the fluorescence field that is emitted by a superconducting qubit~\cite{CampagneEtAlPRL2014}. The statistics of the measurement outcomes generated by this  system cannot be described by  classical probabilities since the density operators at various times do not commute. As far as we know, this is the first time that  a particle quantum filter is applied to an experiment~\cite{Campagne2015} whose measurement statistics are ruled by non-commutative quantum   probabilities.

Section~\ref{sec:discrete} is devoted to the discrete-time formulation. The specific structure of Markov models describing open-quantum systems is presented. Then particle quantum filters are detailed and shown to be always stable (theorem~\ref{thm:stabilityGeneral}). Finally, the link with MaxLike approach and the case of multiple measurement records are addressed. In section~\ref{sec:continuous}, a positivity preserving formulation of particle quantum filters  is proposed for diffusive systems. The mathematical justifications of this formulation is  given in lemma~\ref{lem:SMEfilterPiPartialKraus}. In section~\ref{sec:experiment}, the  numerical  algorithm underlying lemma~\ref{lem:SMEfilterPiPartialKraus}  is   applied  on experimental data  from which  the detection efficiency  is estimated and compared to an existing calibration protocol.

\section{Discrete-time formulation} \label{sec:discrete}

\subsection{Markov models}
In the sequel, $\mathcal{H}$ is the finite-dimensional Hilbert space of the system and expectation values are denoted by the symbol $\EE{.}$.
In this section, time is indexed by the integer $k=0,1,\ldots$
The measurement outcome  at $k$  is  denoted by $y_k$. It corresponds to a classical  output signal.   We limit ourselves to the case where each $y_k$ can take a finite set of values $y_k\in\{1,\ldots,m\}$, $m$ being a positive integer (for continuous values of $y$,  see  section~\ref{sec:continuous}). We denote by $\rho_k$ the density operator at time-step $k$ (an Hermitian operator on
$\mathcal{H}$ such that $\tr{\rho_k} = 1$, $\rho_k \ge 0$).  It  corresponds  to the conditional quantum  state at time $k$ knowing the initial condition $\rho_0$  and the past outcomes $y_{1},\ldots,y_{k}$. According to the law of quantum mechanics, $\rho_{k}$ is related  to $\rho_{k-1}$ via the following Markov process (see, e.g.,~\cite{wiseman-milburnBook}) corresponding to a Davies instrument~\cite{daviesBook1976} in a  discrete context:
\begin{equation} \label{eq:markov-chain-r}
  \rho_k = \frac{\K_{y_k}(\rho_{k-1})}{\tr{\K_{y_k}(\rho_{k-1})}}
\end{equation}
where the super-operator $\rho\mapsto \K_{y}(\rho)$ depends on $y$,  is  linear and completely positive. It  admits the following Kraus representation
$ \K_{y}(\rho)= \sum_{\mu} \Mmu^{y}\rho (\Mmu^{y})^\dag$ where the  operators on $\mathcal{H}$, $(\Mmu^y)$, satisfy
$\sum_{\mu,y}(\Mmu^{y})^\dag\Mmu^y =\Id$ with  $\Id$  the identity operator. Moreover the probability $\PP{y_k=y | \rho_0,y_1,\ldots,y_{k-1}}$ to detect $y_k$  knowing the past outcomes  and the initial state $\rho_0$, depends only on $\rho_{k-1}$ (Markov property) and is given by
$$
  \PP{y_k = y ~\big|~ \rho_{k-1}}= \tr{\K_{y}(\rho_{k-1})}
  .
$$
Notice that $\EE{\rho_{k}~\big|~\rho_{k-1}} = \K(\rho_{k-1})$ where $\K(\rho)= \sum_{y} \K_{y}(\rho)=\sum_{\mu,y}\Mmu^y\rho(\Mmu^y)^\dag$ is a Kraus map (a quantum channel)  since it is not only completely positive  but also  trace preserving: $\tr{\K(\rho)}=\tr{\rho}$. In the sequel, $\K_{y}$ is  called a partial Kraus map since it is not trace preserving in general:  $\tr{\K_y(\rho)} \leq \tr{\rho}$. See, e.g., \cite{dotsenko-et-al:PRA09,AminiSDSMR2013A} for a detailed  construction of such $\K_y$  based on positive operator value measures (POVM)  and left stochastic matrices modeling measurement uncertainties and decoherence.

Now, we consider that the partial Kraus maps  $(\K_y)_{y=1,\ldots,m}$ can depend on time $k$, $(K_{y,k})$,   and on some physical parameters, grouped in the scalar or vectorial time-invariant $p$, $(K_{y,k}^p)$, whose exact value $\hp$ may not be known with a sufficient precision, and whose estimation is the subject of this paper. Here, we consider the case where the only reliable resource of information is some independent series of measurement outcomes,  $(y_k)_{k=1,\ldots, T}$, associated to a   \emph{quantum trajectory} of duration $T$.  Starting from the exact quantum state $\rhob_0$ and the exact parameter value
$\hp$, the exact quantum state trajectory  $(\rhob_k)_{k=1,\ldots, T}$ is given by  the following Markov process:
\begin{equation}\label{eq:exact-filter}
\rhob_k = \frac{\K_{y_k, k}^{\hp}(\rhob_{k-1})}{\tr{\K_{y_k, k}^{\hp}(\rhob_{k-1})}}
\end{equation}
with the following  probability of outcome $y_k$  knowing $\rhob_{k-1}$ and $\hp$:
$$
\PP{y_k = y ~\big|~ \rhob_{k-1},\hp } = \tr{\K_{y, k}^{\hp}(\rhob_{k-1})}
.
$$

\subsection{ Particle quantum  filters}

The parameter estimation method described  in~\cite{GambeW2001PRA,ChaseG2009PRA,NegreM2013NJoP} for continuous-time quantum trajectories admits the following discrete-time formulation.
When the exact parameter value $\hp$ and the initial state $\overline{\rho}_0$ are unknown, one can still resort  to the approximate filter corresponding to its a priori estimate value $p$, with partial Kraus maps $\K_{y_k, k}^p$, an initial guess for  $\rho_0$ and following states $\rho_k^{p}$ satisfying $\rho_k^{p} = \frac{\K_{y_k, k}^{p}(\rho_{k-1}^{p})}{\tr{\K_{y_k, k}^{p}(\rho_{k-1}^{p})}}$. Here, the measurement outcomes $(y_k)_{k=1,\ldots, T}$  correspond to   the  hidden state Markov chain defined in~\eqref{eq:exact-filter} and involving the actual value $\hp$ of the parameter.

Assume that  the initial information of the  true parameter value $\hp$ is that it can take only  two different values $a$ or $b$.  This initial  uncertainty on the value of $\hp$ can be taken into account by using an extended density operator, denoted $\xi=\text{diag}(\xi^a,\xi^b)$, block diagonal, where the first block $\xi^a$  corresponds to $p = a$, and the second block $\xi^b$  to $p = b$.  The evolution of each block is then handled with the corresponding partial  Kraus maps $(\K^a_{y, k})$ and $(\K^b_{y, k})$ forming  extended partial  Kraus maps $\X_{y,k}= \text{diag}\big(\K^a_{y, k}, \K^b_{y, k}\big)$ between   block diagonal density operators on the Hilbert space $\mathcal{H}\times\mathcal{H}$:
\begin{equation}\label{eq:Xi_def}
\X_{y, k}:~\xi
\mapsto  \text{diag}(\K^a_{y, k}(\xi^a), \K^b_{y, k}(\xi^b))
.
\end{equation}
The associated extended quantum filter reads:
\begin{equation}\xi_k = \frac{\X_{y_k,k}(\xi_{k-1})}{\tr{\X_{y_k,k}(\xi_{k-1})}}.\label{eq:Q_filter}\end{equation}

For $p\in\{a,b\}$, the probability that $p=\hp$ at step $k$ knowing the initial quantum  state $\overline{\rho}_0$ and  initial parameter probability $(\pi_0^a,\pi_0^b)$ reads $\pi_k^p=\tr{\xi^p_k}$. Indeed, $\pi_k^a+\pi_k^b=1$ since $\tr{\xi}=\tr{\xi^a} + \tr{\xi^b} =1$, and  $\xi_0=\text{diag}(\pi_0^a\overline{\rho}_0,\pi_0^b\overline{\rho}_0) $. If the  initial information on the parameter value is only  its belonging to $\{a,b\}$, then  $\pi_0^a=\pi_0^b=1/2$.

Instead of using  $\xi=\text{diag}(\xi^a,\xi^b)$ itself, we decompose its terms into products of probabilities $\pi^p$ and density operators $\rho^p=\xi^p/\pi^p$. Then Eq.~(\ref{eq:Q_filter}) reads
\begin{equation} \label{eq:ParticleFilter}
  \left\{
   \begin{array}{l}
   \rho_k^p = \frac{\K_{y_k,k}^p(\rho_{k-1}^p)}{\tr{\K_{y_k,k}^p(\rho_{k-1}^p)}}
   \\
   \pi_k^p = \frac{\tr{\K_{y_k,k}^p(\rho_{k-1}^p)}~ \pi_{k-1}^p}
      {\sum_{p' \in\{a,b\}} \tr{\K_{y_k,k}^{p'}(\rho_{k-1}^{p'})}~ \pi_{k-1}^{p'} }
   \end{array}
  \right.
\end{equation}
for $p\in\{a,b\}$. In the sequel, we will identify  the filter state $\xi$ with  $(\rho^a,\rho^b,\pi^a,\pi^b)$.

We have the following stability result based on~\cite{Rouch2011ACITo,somaraju-et-al:acc2012} and relying on the fidelity $F(\rho,\rho')\in[0,1]$ between two density operators $\rho$ and $\rho'$ defined here as the square of the usual fidelity function used in quantum information~\cite{nielsen-chang-book}:
$$
F(\rho,\rho')=  \trr{\sqrt{\sqrt{\rho}\rho'\sqrt{\rho}}}.
$$
\begin{thm}\label{thm:stability}
Take an arbitrary initial quantum state $\overline{\rho}_0$ and a parameter value $\hp$. Consider the quantum Markov process~\eqref{eq:exact-filter} producing  the measurement record $y_{k}$, $k\geq 0$. Assume that the constant  parameter $\hp$ can only take two different values, $a$ and $b$.  Consider  the  particle  (quantum) filter~\eqref{eq:ParticleFilter} initialized with  $\rho^a_0=\rho^b_0=\rho_0$ ($\rho_0$ any density operator)  and $(\pi^a_0,\pi^b_0)\in[0,1]^2$ with $\pi^a_0+\pi^b_0=1$.
Then  $F(\rhob,\rho^{\hp})$ and $\pi^{\hp}F(\rhob,\rho^{\hp})$ are sub-martingales of the Markov process~\eqref{eq:exact-filter} and~\eqref{eq:ParticleFilter}
of state
$(\rhob,\rho^a,\rho^b,\pi^a,\pi^b)$:
\end{thm}
When $\rho_0=\rhob_0$,  we have $\rho^{\hp}\equiv\rhob$,
$F(\rhob,\rho^{\hp})=1$. Thus $\pi^{\hp}$ is a sub-martingale
$$
\EE{\pi^{\hp}_k~\big|~\rhob_{k-1},\xi_{k-1}} \geq \pi^{\hp}_{k-1}
$$
This means that, in practice, the component of $\pi$ associated to the true value of the parameter tends to increase.
\begin{proof}
The fact that $F(\rhob,\rho^{\hp})$ is a sub-martingale is a direct consequence of~\cite[theorem IV.1]{somaraju-et-al:acc2012}: $(\rhob,\rho^{\hp})$ is the state of the following quantum Markov chain
 $$
 \rhob_k = \frac{\K_{y_k, k}^{\hp}(\rhob_{k-1})}{\tr{\K_{y_k, k}^{\hp}(\rhob_{k-1})}}
 , \quad
 \rho^{\hp}_k = \frac{\K_{y_k, k}^{\hp}( \rho^{\hp}_{k-1})}{\tr{\K_{y_k, k}^{\hp}( \rho^{\hp}_{k-1})}}
 $$
 with initial state $(\rhob_0,\rho_0)$ and  measurement outcome $y_k$  whose  probability
 $\PP{y_k = y ~\big|~ \rhob_{k-1} } = \tr{\K_{y, k}^{\hp}(\rhob_{k-1})}$ depends only on $\rhob_{k-1}$.

 For instance, assume that $\hp=a$. Denote by $\xib$ the  state of the quantum filter~\eqref{eq:Q_filter} initialized with $\xib_0=\text{diag}(\rhob_0,0)$. Then
 $\xib\equiv(\rhob,0)$  and thus $(\xib,\xi)$ is solution of the  extended Markov chain
 $$
 \xib_k = \frac{\X_{y_k, k}(\xib_{k-1})}{\tr{\X_{y_k, k}(\xib_{k-1})}}
 , \quad
 \xi_k = \frac{\X_{y_k, k}( \xi_{k-1})}{\tr{\X_{y_k, k}( \xi_{k-1})}}
 $$
 with   measurement outcome $y_k$  of  probability
 $\PP{y_k = y ~\big|~ \xib_{k-1} } = \tr{\X_{y, k}(\xib_{k-1})}$ depending only on $\xib_{k-1}$.
Thus according to~\cite[theorem IV.1]{somaraju-et-al:acc2012}, $F(\xib,\xi)$ is a sub-martingale. Due to the block structure of $\xib=\text{diag}(\rhob,0)$ and
$\xi=\text{diag}(\pi^a\rho^a,\pi^b\rho^b)$, we have $F(\xib,\xi)=\pi^a F(\rhob,\rho^a)$.
\end{proof}

Extension of theorem~\ref{thm:stability} to an arbitrary number $r$ of parameter values  is given below, the proof being very similar and not detailed here.
\begin{thm}\label{thm:stabilityGeneral}
Take an arbitrary initial quantum state $\overline{\rho}_0$ and parameter value $\hp$. Consider the quantum Markov process~\eqref{eq:exact-filter} producing  the measurement record $y_{k}$, $k\geq 0$. Assume that the   parameter $\hp$  belongs to a set of  $r$ different values $(p_l)_{l=1,\ldots,r}$. Take, for $l=1,\ldots,r$,  the  particle  quantum  filter
\begin{equation*}
  \left\{
   \begin{array}{l}
   \rho_k^{p_l} = \frac{\K_{y_k,k}^{p_l}(\rho_{k-1}^{p_l})}{\tr{\K_{y_k,k}^{p_l}(\rho_{k-1}^{p_l})}}
   \\
   \pi_k^{p_l} = \frac{\tr{\K_{y_k,k}^{p_l}(\rho_{k-1}^{p_l})}~ \pi_{k-1}^{p_l}}
      {\sum_{j=1}^r \tr{\K_{y_k,k}^{{p_{j}}}(\rho_{k-1}^{{p_{j}}})}~ \pi_{k-1}^{{p_{j}}} }
   \end{array}
  \right.
\end{equation*}
initialized with  $\rho^{{p_l}}_0=\rho_0$ ($\rho_0$ any density operator)  and $(\pi^{p_1}_0,\ldots,\pi^{p_r}_0)\in[0,1]^r$ with $\sum_{j} \pi^{p_{j}}_0=1$.

Then  $F(\rhob,\rho^{\hp})$ and $\pi^{\hp}F(\rhob,\rho^{\hp})$ are sub-martingales of the Markov process driven  by~\eqref{eq:exact-filter} and of
state
$(\rhob,\rho^{p_1},\ldots,\rho^{p_r},\pi^{p_1},\ldots,\pi^{p_r})$:
\end{thm}

Extension  to a continuum of values for $p$ of such particle  quantum filters and of the above stability result  can be done without major difficulties.

\subsection{Connexion with MaxLike methods}
 Assume that the initial  density operator is well known: $\rho_0=\rhob_0$. It is possible to choose as an estimation of $\hp$, among $a$ or $b$, the value $p$ that maximises the probability $\pi_k^p$ after a certain amount of time $k$. This method is actually a \emph{maximum-likelihood} based technique. The multiplicative increment at time $k$ for $\pi_k^a$ is $\tr{\K^a_{y_k, k}(\rho_{k-1}^a)}$, which is equal to $\PP{y_k~\Big|~\rho_0, y_1, \ldots, y_{k-1}, \hp = a}$.
From this observation, we deduce that
\begin{equation*}
\pi_k^a = \frac{\pi_0^a}{C_k} \times
\prod_{l = 1}^{k} \PP{y_l ~\Big|~\rho_0, y_1, \ldots, y_{l-1}, \hp = a},
\end{equation*}
where  $C_k$ is a normalization factor to ensure $\pi_k^a + \pi_k^b=1$. Remarking that the probability of the measurement outcomes  $(y_l)_{l \le k}$ is the probability of the measurement outcomes  $(y_l)_{l \le k-1}$ times the probability of $y_k$ conditionally to all prior measurements, one gets
\begin{equation*}
\pi_k^a = \frac{\pi_0^a}{C_k} \times
\PP{y_1, \ldots, y_{k} ~\Big|~\rho_0, \hp = a},
\end{equation*}
and similarly
\begin{equation*}
\pi_k^b = \frac{\pi_0^b}{C_k} \times
\PP{y_1, \ldots, y_{k} ~\Big|~\rho_0, \hp = b}.
\end{equation*}
Choosing as an estimate the value $a$ or $b$ whose associated component of $\pi$   tends towards $1$ thus amounts to choosing the parameter value that maximises the probability of the measurement outcomes $(y_1,\ldots,y_T)$.

\subsection{Multiple quantum trajectories} \label{ssec:multipleTrajectories}

Such particle quantum filtering techniques extend without difficulties   to $\nmax$ records (indexed by $n\in\{1,\ldots \nmax\}$) of measurement outcomes,   $(y_k^{(n)})_{k=1,\ldots, T_n}$ with possibly  different lengths $T_n$ and initial conditions $\rhob^{(n)}_0$. This extension consists in a concatenation of  the  $\nmax$ records into a single record $(\bar y_k)_{k=1,\ldots,T}$ with $T=\sum_{n=1}^{\nmax} T_n$ and
\begin{multline*}
 (\bar y_k)_{k=1,\ldots,T}=\\
 \Big( y_1^{(1)},\ldots,y_{T_1}^{(1)}, y_1^{(2)},\ldots,y_{T_2}^{(2)}, \quad \ldots \quad,
   y_1^{(\nmax)},\ldots,y_{T_{\nmax}}^{(\nmax)} \Big)
\end{multline*}
This record can be  associated to  a  single quantum trajectory of length $T$ of form~\eqref{eq:exact-filter}. First  initialize at $\rhob_0^{(1)}$. Then  for each $k$ equal to $T_1+\ldots+T_{n-1}$,  $\rho_{k+1}$ is reset to $\rhob_0^{(n)}$. This  can be done  by applying a reset  Kraus  map $\K^{\rhob_0^{(n)}}$   after the computation of $\rhob_{k+1}$ relying on outcome  $y_{T_{n-1}}^{(n-1)}$ and before  using  the outcome $y_1^{(n)}$.  For any density operator $\sigma$, it is simple to construct via its  spectral decomposition,  a Kraus map $\K^{\sigma}$ such that, for all density operator $\rho$, $\K^{\sigma}(\rho)=\sigma$.  With this trick $(\bar y_k)$ is associated to an effective single quantum trajectory of the form~\eqref{eq:exact-filter} where the partial Kraus maps $\K_{y,k}^{\hp}$ depend effectively  on the time step $k$   because of adding   these reset Kraus maps.

For the   particle quantum filter that is described in theorem~\ref{thm:stabilityGeneral} and associated to the record $(\bar y_k)$ ,   each $\rho^{(p_l)}_{k}$ is reset in a similar way at each time step $k=T_1+\ldots+T_{n-1}$  contrarily to  the   parameter probability $\pi^{(p_l)}_{k} $ that is not reset.

\section{Continuous-time  formulation} \label{sec:continuous}

\subsection{Diffusive stochastic master equations}
For a mathematical  and precise   description of  such diffusive models, see~\cite{BarchielliGregorattiBook}. We just recall here   the stochastic  master equation governing the  time evolution of the density operator $t\mapsto \rho_t$
\begin{multline}\label{eq:SME}
\dif{\rho}_t  = \Big(-i[H, \rho_t] + \sum_{\nu=1}^m \mathcal{D}_\nu(\rho_t)\Big)\dif{t}\\
+ \sum_{\nu=1}^m \sqrt{\eta_\nu}~\Big( L_\nu \rho_t+ \rho_t L_\nu^\dag - \tr{L_\nu \rho_t+ \rho_t L_\nu^\dag} \rho_t\Big)\dif{W}_t^\nu
\end{multline}
where $H$ is the Hamiltonian, an Hermitian operator on $\mathcal{H}$ ($\hbar=1$ here) and where,  for each $\nu\in\{1,\ldots,m\}$,
\begin{itemize}

\item $\mathcal{D}_\nu$ is the  Lindblad super-operator
$$
\mathcal{D}_\nu(\rho)= L_\nu \rho L_\nu^\dagger - \tfrac{1}{2}(L_\nu^\dagger L_\nu \rho + \rho L_\nu^\dagger L_\nu) ;
$$

\item $L_\nu$ is an operator on $\mathcal{H}$, which is not necessarily Hermitian and which is associated to the measurement/decoherence channel  $\nu$ ;

\item $\eta_\nu\in[0,1]$ is the detection efficiency ($\eta_\nu=0$ for decoherence channel and $\eta_\nu>0$ for measurement channel) ;

\item $W_t^\nu$ is a Wiener process (independent of the other Wiener processes $W_t^{\mu\neq\nu}$) describing the quantum fluctuations of the  continuous  output signal  $t\mapsto y^\nu_t$. It is related to $\rho_t$ by
\begin{equation}\label{eq:SMEoutput}
  \dif{y}_t^\nu = \sqrt{\eta_\nu}~\tr{L_\nu \rho_t+ \rho_t L_\nu^\dag}~ \dif{t} + \dif{W}_t^\nu
.
\end{equation}

\end{itemize}

\subsection{Partial Kraus map formulation}

We introduce here another formulation of~\eqref{eq:SME} that mimics the discrete-time formulation~\eqref{eq:exact-filter}.
 This formulation is inspired of   subsection 4.3.3 of~\cite{haroche-raimondBook06}, subsection  entitled "Physical interpretation of the master equation".
 In~\eqref{eq:SME}, $\dif{\rho}_t$ stands for $\rho_{t+\dif{t}} -\rho_{t}$. It can thus be written as
\begin{multline*}
\rho_{t+\dif{t}}  = \rho_t +  \Big(-i[H, \rho_t] + \sum_{\nu=1}^m \mathcal{D}_\nu(\rho_t)\Big)\dif{t}\\
+ \sum_{\nu=1}^m \sqrt{\eta_\nu}~\Big( L_\nu \rho_t+ \rho_t L_\nu^\dag - \tr{L_\nu \rho_t+ \rho_t L_\nu^\dag} \rho_t\Big)\dif{W}_t^\nu
\end{multline*}
i.e., $\rho_{t+\dif{t}}$ is an algebraic expression involving $\rho_t$, $\dif{t}$ and $\dif{W}_t^\nu$.
  With this form, it is not obvious that $\rho_{t+\dt}$ remains a density operator if $\rho_t$ is a density operator. The following lemma provides another formulation based on  \Ito~calculus showing directly that $\rho_{t+\dt}$ remains a density operator. In~\cite{RouchR2015PRA}, similar   formulations  are   proposed without the mathematical  justifications given below  and are  tested in  realistic simulations  of  measurement-based feedback scheme.

\begin{lem}\label{lem:SMEPartialKraus}
Consider the stochastic differential equation~\eqref{eq:SME} with an initial condition $\rho_0$, which is a non-negative  Hermitian operator  of trace one.
Then it also reads:
$$
\rho_{t+\dif{t}} = \frac{\K_{dy_t,dt}(\rho_t)}{\tr{\K_{dy_t,dt}(\rho_t)}},
$$
where $dy_t$ stands for $(dy^1_t,\ldots,dy^m_t)$, and where $\K_{\Dy,\Dt}$ is a partial Kraus map depending on  $\Dy\in\RR^m$ and $\Dt >0$ given by
$$
\K_{\Dy,\Dt}(\rho)= M_{\Dy,\Dt}~\rho~ M_{\Dy,\Dt}^\dag + \sum_{\nu=1}^m (1-\eta_\nu) \Dt ~L_\nu \rho L_\nu^\dag
$$
and $M_{\Dy,\Dt}$ is the following operator on $\mathcal{H}$
$$
M_{\Dy,\Dt}= \Id  - \left(iH + \sum_{\nu =1}^{m} L_\nu^\dagger L_\nu /2\right) \Dt + \sum_{\nu=1}^m \sqrt{\eta_\nu} \Dy^\nu~ L_\nu
$$
\end{lem}

\begin{proof}
  Assume  that  $m=1$. Then,
  \begin{multline} \label{eq:SME-bis}
   \dif{\rho}_t  = \Big(-i[H, \rho_t] +  L\rho_t L^\dag - \tfrac{1}{2}(L^\dag L\rho_t + \rho_t L^\dag L)\Big)\dif{t}\\
+ \sqrt{\eta}~\Big( L\rho_t+ \rho_t L^\dag - \tr{L \rho_t+ \rho_t L^\dag} \rho_t\Big)\dif{W}_t
.
  \end{multline}
Using  It$\bar{\text{o}}$  rules, $\dy_t^2=\dt$. Hence, we have
\begin{multline*}
\K_{\dy_t,\dt}(\rho_t) =
\rho_t+  \sqrt{\eta}  (L\rho_t+\rho_t L^\dag) ~\dy_t
 \\
 +  \left( -i[H,\rho_t] + L\rho_t L^\dag  - \tfrac{1}{2} (L^\dag L\rho_t + \rho_t L^\dag L ) \right) ~ \dt
 .
\end{multline*}
Thus $\tr{\K_{\dy_t,\dt}(\rho_t)}= 1 +  \sqrt{\eta}  \tr{L\rho_t+\rho_t L^\dag} \dy_t$ and
\begin{multline*}
\frac{1}{\tr{\K_{\dy_t,\dt}(\rho_t)}}
=1- \sqrt{\eta}  \tr{L\rho_t+\rho_t L^\dag} \dy_t
\\+ \eta \trr{L\rho_t+\rho_t L^\dag} \dt
.
\end{multline*}
We get
\begin{multline*}
\frac{\K_{\dy_t,\dt}(\rho_t)}{\tr{\K_{\dy_t,\dt}(\rho_t)}} - \rho_t
 \\
 =  \sqrt{\eta}~\Big( L\rho_t+ \rho_t L^\dag - \tr{L \rho_t+ \rho_t L^\dag} \rho_t\Big) ~\dy_t
 \\
 + \left( -i[H,\rho_t] + L\rho_t L^\dag  - \tfrac{1}{2} (L^\dag L\rho_t + \rho_t L^\dag L ) \right) ~ \dt
 \\
 - \eta\tr{L \rho_t+ \rho_t L^\dag} \Big( L\rho_t+ \rho_t L^\dag - \tr{L \rho_t+ \rho_t L^\dag} \rho_t\Big) ~\dt
.
\end{multline*}
One recognizes~\eqref{eq:SME-bis} since  $\dy_t -\sqrt{\eta}\tr{L \rho_t+ \rho_t L^\dag}\dt = \dif{W}_t$.
For  $m>1$,  the computations are similar  and  not  detailed here.
\end{proof}

\subsection{Particle quantum filtering}

Assume  the system dynamics depends on a constant  parameter $p$ appearing either in the SME~\eqref{eq:SME} and/or in the output maps~\eqref{eq:SMEoutput}. As in section~\ref{sec:discrete}, assume that $p$ can take a finite number $r$ of values $p_1$, \ldots, $p_r$. Denote by $\rho^p_t$ the quantum state associated to~$p$:
\begin{equation}\label{eq:SMEp}
\dif{\rho}_t^p = \mathcal{L}^p(\rho^p_t) \dif{t}
+ \sum_{\nu=1}^m \mathcal{M}^p(\rho_t^p) \dif{W}_t^\nu
\end{equation}
where  the super-operators
\begin{multline*}
\mathcal{L}^p(\rho)=
  -i[H^p, \rho] \\
  + \sum_{\nu=1}^m L_\nu^p \rho (L_\nu^p)^\dagger - \frac{1}{2}((L_\nu^p)^\dagger L^p_\nu \rho + \rho (L^p_\nu)^\dagger L^p_\nu)
\end{multline*}
and
\begin{multline*}
\mathcal{M}^p(\rho)=
  \sqrt{\eta^p_\nu}~\Big( L^p_\nu \rho+ \rho (L^p_\nu)^\dag - \tr{L^p_\nu \rho+ \rho (L^p_\nu)^\dag} \rho\Big)
\end{multline*}
depend on $p$ since the operators $L^p_\nu$ and the efficiencies $\eta_\nu^p$ could  depend  on $p$.
The  $m$ outputs that are associated to the parameter $p$ then read:
\begin{equation}\label{eq:SMEoutputp}
  dy_t^\nu = C^p_\nu(\rho_t^p) \dt + \dif{W}_t^\nu
\end{equation}
for $\nu=1,\ldots, m$, and where:
$$
 C^p_\nu(\rho) =  \sqrt{\eta^p_\nu}\tr{L^p_\nu \rho+ \rho (L^p_\nu)^\dag}
 .
 $$
With these notations, the particle quantum filter introduced in~\cite{GambeW2001PRA} and further developed and  analyzed in~\cite{ChaseG2009PRA,NegreM2013NJoP}  reads as follows.
For each $l\in\{1,\ldots,r\}$, $\rho^{p_l}_t$ is governed by the quantum filter:
\begin{multline}\label{eq:SMEfilter}
\dif{\rho}_t^{p_l} = \mathcal{L}^{p_l}(\rho^{p_l}_t) \dif{t}
\\+ \sum_{\nu=1}^m \mathcal{M}^{p_l}(\rho_t^{p_l})~ \big(dy_t^\nu -C^{p_l}_\nu(\rho_t^{p_l}) \dt\big),
\end{multline}
and the parameter probability $\pi_t^{p_l}$ is governed by:
\begin{multline}\label{eq:SMEPi}
\dif{\pi}_t^{p_l} = \pi^{p_l}_t
   \left( \sum_{\nu=1}^m
   \left( C^{p_l}_\nu(\rho_t^{p_l})- \overline{C}^\nu_t \right)
   \left(dy_t^\nu - \overline{C}^\nu_t \dt\right)
   \right),
\end{multline}
where
$\overline{C}^\nu_t= \sum_{j=1}^r
 \pi^{p_{j}}_t  C^{p_{j}}_\nu(\rho_t^{p_{j}})$.

Here again, the lemma below provides another formulation of this particle quantum filter that mimics the discrete-time setting of theorem~\ref{thm:stabilityGeneral}.

\begin{lem}\label{lem:SMEfilterPiPartialKraus}
For each $l\in\{1,\ldots,r\}$, the  particle quantum filter~\eqref{eq:SMEfilter} and~\eqref{eq:SMEPi} can be formulated as follows:
\begin{equation*}
\left\{
   \begin{array}{rl}
 \rho^{p_l}_{t+\dif{t}} &= \frac{\K^{p_l}_{dy_t,dt}(\rho^{p_l}_t)}{\tr{\K^{p_l}_{dy_t,dt}(\rho^{p_l}_t)}}
 \\
 \pi^{p_l}_{t+\dif{t}} &= \frac{\tr{\K^{p_l}_{dy_t,dt}(\rho^{p_l}_t)}~ \pi_{t}^{p_l}}
      {\sum_{j=1}^r\tr{\K^{p_{j}}_{dy_t,dt}(\rho^{p_{j}}_t)}~ \pi_{t}^{p_{j}} }
\end{array}
\right.
\end{equation*}
where $dy_t$ stands for $(dy^1_t,\ldots,dy^m_t)$ and where $\K^{p}_{\Dy,\Dt}$ is a partial Kraus map depending on $p$,  $\Dy\in\RR^m$ and $\Dt >0$ given by:
$$
\K^{p}_{\Dy,\Dt}(\rho)= M^{p}_{\Dy,\Dt} ~\rho~ \left(M^{p}_{\Dy,\Dt}\right)^\dag + \sum_{\nu=1}^m (1-\eta^{p}_\nu) \Dt ~L^{p}_\nu \rho (L_\nu^{p})^\dag,
$$
and $M^{p}_{\Dy,\Dt}$ is the following operator on $\mathcal{H}$:
$$
M^{p}_{\Dy,\Dt}= \Id  - \left(iH^{p} + \sum_{\nu =1}^{m} (L^{p}_\nu)^\dagger L^{p}_\nu /2\right) \Dt + \sum_{\nu=1}^m \sqrt{\eta^{p}_\nu} \Dy^\nu~ L^{p}_\nu.
$$
\end{lem}
The proof  is very similar to the proof of lemma~\ref{lem:SMEPartialKraus}. It relies on simple but slightly tedious computations exploiting \Ito~calculus. Due to space limitation, this proof is not detailed here.  This lemma, combined with the mathematical machineries  exploited in~\cite{AminiPR2014}, opens  the way to an  extension to  the diffusive case  of   theorem~\ref{thm:stabilityGeneral}.

\section{An experimental validation} \label{sec:experiment}

The estimation of the detection efficiency  is conducted on a superconducting qubit  whose fluorescence field is measured using a heterodyne detector~\cite{Roch2012,Hatridge2013}. For the detailed  physics of this  experiment, see~\cite{CampagneEtAlPRL2014, Campagne2015}. The  Hilbert space $\mathcal{H}$ is $\mathbb{C}^2$. The system  dynamics is   described by a  stochastic master equation of the form~\eqref{eq:SME}, with $m=3$: $\eta_1=\eta_2=\eta$ is the total efficiency of the heterodyne measurement of the fluorescence signal; $\eta_3=0$ corresponds to an unmonitored dephasing channel:
$$
L_1= \sqrt{\tfrac{1}{2T_1}} \frac{X-iY}{2}, \quad  L_2=i L_1, \quad L_3=\sqrt{\tfrac{1}{2T_\phi}} Z
$$
where $X$,  $Y$ and $Z$ are the usual Pauli  matrices~\cite{nielsen-chang-book}. The time constants $T_1=4.15~\mu$s and $T_\phi=35~\mu$s are  determined independently using Rabi or Ramsey protocols, which is not the case of $\eta$. Using a calibration of the average resonance fluorescence signal, the measured vacuum noise fluctuations provide a first estimation of $\eta=0.26\pm0.02$.

To get a more precise estimation of $\eta$, we have measured $\nmax=3\times 10^6$ quantum trajectories of $10~\mu$s, starting from the same known initial state $\rhob_0=\frac{\Id+X}{2}$. The sampling time $\Dt$ is equal to $0.20~\mu$s. For each trajectory,   the  measurement sample at  time $t_k= k\Dt$, $k\in\{1,\ldots, 50\}$,  corresponds to the two quadratures of the fluorescence field $\Dy^1_k=y^1_{k\Dt}-y^{1}_{(k-1)\Dt}$ and $\Dy^2_k=y^2_{k\Dt}-y^{2}_{(k-1)\Dt}$.  From  lemma~\ref{lem:SMEfilterPiPartialKraus}, we derive a simple recursive algorithm where  $(\dy_t)$ and $\dt$ are replaced by
$(\Dy_k)$ and $\Dt$.  Moreover, as explained in subsection~\ref{ssec:multipleTrajectories}, the $3\times 10^6$ quantum  trajectories  are concatenated into a single one.

The estimation is done by taking some pattern values $\eta_1$, $\eta_2$, ..., $\eta_r$,  assuming that the real value $\heta$ is sufficiently close to one of them. We begin with a first estimation that keeps a big interval between each possible value $\eta_i$ of $\eta$, in order to validate our estimation scheme. We then sharpen this estimation by reducing the intervals between each value $\eta_i$, until arriving to a level of accuracy after which no distinct discrimination can be performed.
The results are given at figures~\ref{fig:fig1a}, \ref{fig:fig1b} and \ref{fig:fig1c}. They give  the following refinement of the initial calibration: $\heta = 0.2425 \pm 0.005$. On each of the figures, the X-axis represents the number of trajectories after which we look at the parameter probabilities $\pi_k^{\eta_i}$ and the Y-axis displays these probabilities.

\begin{figure}
	\includegraphics[width=\columnwidth]{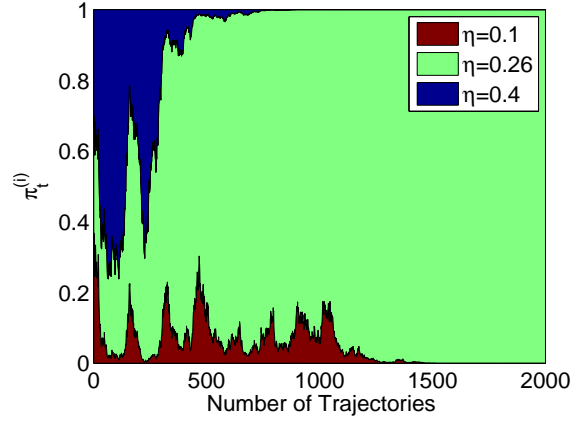}
	\caption{First estimation, with pattern values $\eta_1 = 0.10$, the parameter value $\eta_2 = 0.26$ close to $\heta$, and $\eta_3 = 0.40$. Only the first $2000$ trajectories are needed  to select $\heta\approx 0.26$ and discard $0.10$ and $0.40$. }
	\label{fig:fig1a}
\end{figure}
\begin{figure}
		\includegraphics[width=\columnwidth]{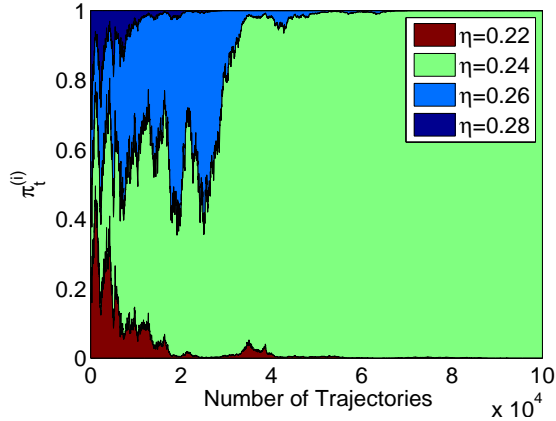}
		\caption{Second estimation, realized with more narrow intervals between each pattern values. We notice that $\heta$ is actually closer to $0.24$ than $0.26$, the calibrated  value, and that the number of trajectories required for the discrimination has drastically increased to $1\times 10^5$.}
		\label{fig:fig1b}
\end{figure}
\begin{figure}
		\includegraphics[width=\columnwidth]{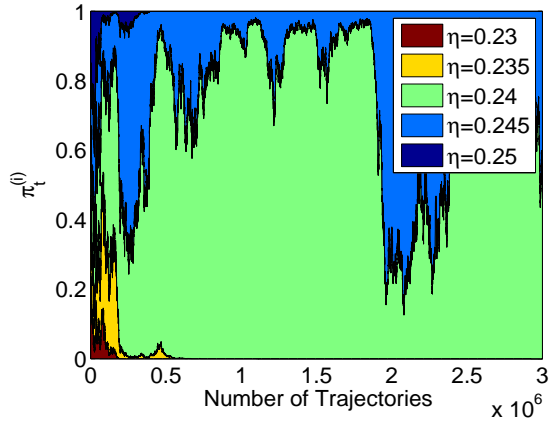}
		\caption{Last estimation, with very narrow intervals. We use all the trajectories available, i.e. $3\times10^6$ trajectories.  Filter does not converge to a distinct choice between $0.240$ and $0.245$.}
		\label{fig:fig1c}
\end{figure}

\section{Conclusion}

 We have shown  that particle quantum filtering is always a stable process. We have   proposed  an original  positivity preserving formulation for systems governed by diffusive stochastic master equation. A first validation on experimental data confirms the interest of the resulting parameter algorithm. This positivity preserving algorithm appears to be robust enough to  cope with sampling time of more than 2\% of  the characteristic time attached to the measurement.
 The convergence  characterization of such estimation scheme remains to be done despite the fact they are always stable.

\section*{Acknowledgment}
The authors  thank Michel Brune,  Igor Dotsenko and  Jean-Michel Raimond for  useful discussions on quantum filtering and  parameter estimation  in the discrete-time case.


\begin{thebibliography}{10}

\bibitem{AminiPR2014}
H.~Amini, C.~Pellegrini, and P.~Rouchon.
\newblock Stability of continuous-time quantum filters with measurement
  imperfections.
\newblock {\em Russian Journal of Mathematical Physics}, 21(3):297--315--,
  2014.

\bibitem{AminiSDSMR2013A}
H.~Amini, R.A. Somaraju, I.~Dotsenko, C.~Sayrin, M.~Mirrahimi, and P.~Rouchon.
\newblock Feedback stabilization of discrete-time quantum systems subject to
  non-demolition measurements with imperfections and delays.
\newblock {\em Automatica}, 49(9):2683--2692, September 2013.

\bibitem{BarchielliGregorattiBook}
A.~Barchielli and M.~Gregoratti.
\newblock {\em Quantum Trajectories and Measurements in Continuous Time: the
  Diffusive Case}.
\newblock Springer Verlag, 2009.

\bibitem{Brakhane2012}
S.~Brakhane, Alt W., T.Kampschulte, M.~Martinez-Dorantes, R.~Reimann, S.~Yoon,
  A.~Widera, and D.~Meschede.
\newblock Bayesian feedback control of a two-atom spin-state in an atom-cavity
  system.
\newblock {\em Phys. Rev. Lett.}, 109(17):173601--, October 2012.

\bibitem{CampagneEtAlPRL2014}
P.~Campagne-Ibarcq, L.~Bretheau, E.~Flurin, A.~Auff\`eves, F.~Mallet, and
  B.~Huard.
\newblock Observing interferences between past and future quantum states in
  resonance fluorescence.
\newblock {\em Phys. Rev. Lett.}, 112:180402, May 2014.

\bibitem{Campagne2015}
P.~Campagne-Ibarcq {\em et al.},
\newblock in preparation.

\bibitem{MoulinesBook2005}
O.~Capp\'{e}, E.~Moulines, and T.~Ryden.
\newblock {\em Inference in Hidden Markov Models}.
\newblock Springer series in statistics, 2005.

\bibitem{ChaseG2009PRA}
Bradley~A. Chase and J.~M. Geremia.
\newblock Single-shot parameter estimation via continuous quantum measurement.
\newblock {\em Phys. Rev. A}, 79(2):022314--, February 2009.

\bibitem{daviesBook1976}
E.B. Davies.
\newblock {\em Quantum Theory of Open Systems}.
\newblock Academic Press, 1976.

\bibitem{dotsenko-et-al:PRA09}
I.~Dotsenko, M.~Mirrahimi, M.~Brune, S.~Haroche, J.-M. Raimond, and P.~Rouchon.
\newblock Quantum feedback by discrete quantum non-demolition measurements:
  towards on-demand generation of photon-number states.
\newblock {\em Physical Review A}, 80: 013805-013813, 2009.

\bibitem{GambeW2001PRA}
J.~Gambetta and H.~M. Wiseman.
\newblock State and dynamical parameter estimation for open quantum systems.
\newblock {\em Phys. Rev. A}, 64(4):042105--, September 2001.

\bibitem{haroche-raimondBook06}
S.~Haroche and J.M. Raimond.
\newblock {\em Exploring the Quantum: Atoms, Cavities and Photons.}
\newblock Oxford University Press, 2006.

\bibitem{Hatridge2013}
M.~Hatridge {\em et al.},
\newblock Quantum Back-Action of an Individual Variable-Strength Measurement.
\newblock {\em Science}, 339:178 (2013)

\bibitem{KatoY2013}
Y.~Kato and N.~Yamamoto.
\newblock Estimation and initialization of quantum network via continuous
  measurement on single node.
\newblock In {\em Decision and Control (CDC), 2013 IEEE 52nd Annual Conference
  on}, pages 1904--1909, 2013.

\bibitem{Mabuc1996QaSOJotEOSPB}
H~Mabuchi.
\newblock Dynamical identification of open quantum systems.
\newblock {\em Quantum and Semiclassical Optics: Journal of the European
  Optical Society Part B}, 8(6):1103--, 1996.

\bibitem{NegreM2013NJoP}
A~Negretti and K~M{\o}lmer.
\newblock Estimation of classical parameters via continuous probing of
  complementary quantum observables.
\newblock {\em New Journal of Physics}, 15(12):125002--, 2013.

\bibitem{nielsen-chang-book}
M.A. Nielsen and I.L. Chuang.
\newblock {\em Quantum Computation and Quantum Information}.
\newblock Cambridge University Press, 2000.

\bibitem{Roch2012}
N. Roch {\em et al.},
\newblock Widely Tunable, Nondegenerate Three-Wave Mixing Microwave Device Operating near the Quantum Limit
\newblock {\em Physical Review Letters}, 108:147701 (2012)

\bibitem{Rouch2011ACITo}
P.~Rouchon.
\newblock Fidelity is a sub-martingale for discrete-time quantum filters.
\newblock {\em IEEE Transactions on Automatic Control}, 56(11):2743--2747,
  2011.

\bibitem{RouchR2015PRA}
P.~Rouchon and J.~F. Ralph.
\newblock Efficient quantum filtering for quantum feedback control.
\newblock {\em Phys. Rev. A}, 91(1):012118--, January 2015.

\bibitem{somaraju-et-al:acc2012}
A.~Somaraju, I.~Dotsenko, C.~Sayrin, and P.~Rouchon.
\newblock Design and stability of discrete-time quantum filters with
  measurement imperfections.
\newblock In {\em American Control Conference}, pages 5084--5089, 2012.

\bibitem{wiseman-milburnBook}
H.M. Wiseman and G.J. Milburn.
\newblock {\em Quantum Measurement and Control}.
\newblock Cambridge University Press, 2009.

\end{thebibliography}
\end{document}